\documentclass[12pt,reqno]{amsart}
\usepackage{amsmath}
\usepackage{amssymb}
\usepackage[left=3cm,top=3cm,right=3cm,bottom=3cm]{geometry}
\usepackage{epsfig}
\usepackage{graphicx}
\usepackage[usenames,dvipsnames]{color}

\begin{document}
\newtheorem{theorem}{Theorem}[section]
\newtheorem{lemma}[theorem]{Lemma}
\newtheorem{definition}[theorem]{Definition}
\newtheorem{conjecture}[theorem]{Conjecture}
\newtheorem{proposition}[theorem]{Proposition}
\newtheorem{algorithm}[theorem]{Algorithm}
\newtheorem{corollary}[theorem]{Corollary}
\newtheorem{observation}[theorem]{Observation}
\newtheorem{problem}[theorem]{Open Problem}
\newtheorem{remark}[theorem]{Remark}
\newcommand{\noin}{\noindent}
\newcommand{\ind}{\indent}
\newcommand{\om}{\omega}
\newcommand{\I}{\mathcal I}
\newcommand{\N}{{\mathbb N}}
\newcommand{\LL}{\mathbb{L}}
\newcommand{\R}{{\mathbb R}}
\newcommand{\E}{\mathbb E}
\newcommand{\Prob}{\mathbb{P}}
\newcommand{\eps}{\varepsilon}
\newcommand{\G}{{\mathcal{G}}}
\newcommand{\Bin}{\mathrm{Bin}}

\title{The set chromatic number of random graphs}

\author{Andrzej Dudek}
\address{Department of Mathematics, Western Michigan University, Kalamazoo, MI, USA}
\email{\tt andrzej.dudek@wmich.edu}

\author{Dieter Mitsche}
\address{Universit\'{e} de Nice Sophia-Antipolis, Laboratoire J-A Dieudonn\'{e}, Parc Valrose, 06108 Nice cedex 02}
\email{\texttt{dmitsche@unice.fr}}

\author{Pawe\l{} Pra\l{}at}
\address{Department of Mathematics, Ryerson University, Toronto, ON, Canada}
\email{\tt pralat@ryerson.ca}

\thanks{The research of the first author is supported in part by Simons Foundation Grant \#244712.}
\thanks{The research of the third author is supported in part by NSERC and Ryerson University.}

\maketitle

\begin{abstract}
In this paper we study the set chromatic number of a random graph $\G(n,p)$ for a wide range of $p=p(n)$. We show that the set chromatic number, as a function of $p$, forms an intriguing zigzag shape. 
\end{abstract}

\section{Introduction}\label{sec:intro}

A \textbf{proper colouring} of a graph is a labeling of its vertices with colours such that no two vertices sharing the same edge have the same colour. A colouring using at most $k$ colours is called a \textbf{proper $k$-colouring}. The smallest number of colours needed to colour a graph $G$ is called its \textbf{chromatic number}, and it is denoted by $\chi(G)$. 

In this paper we are concerned with another notion of colouring, first introduced by Chartrand et al.~\cite{CORZ}. For a given (not necessarily proper) $k$-colouring $c:V \to [k]$ of the vertex set of $G=(V,E)$,  let 
$$
C(v) = \{ c(u) : uv \in E \}
$$
be the \textbf{neighbourhood colour set} of a vertex $v$. (In this paper,  $[k]:= \{1, 2, \ldots, k\}$.) The colouring $c$ is a \textbf{set colouring} if $C(u) \neq C(v)$ for every pair of adjacent vertices in $G$. The minimum number of colours, $k$, required for such a colouring is the \textbf{set chromatic number} $\chi_s(G)$ of $G$. One can show that
\begin{equation}\label{eq:chis_trivial}
\log_2 \chi(G) + 1 \le \chi_s(G) \le \chi(G).
\end{equation}
Indeed, the upper bound is trivial, since any proper colouring~$c$ is also a set colouring: for any edge $uv$, $N(u)$, the neighbourhood of $u$, contains $c(v)$ whereas $N(v)$ does not. On the other hand, suppose that there is a set colouring using at most $k$ colours. Since there are at most $2^k$ possible neighbourhood colour sets, one can assign a unique colour to each set obtaining a proper colouring using at most $2^k$ colours. We get that 
$\chi(G) \le 2^{\chi_s(G)}$, or equivalently,  $\chi_s(G) \ge \log_2 \chi(G)$. With slightly more work, one can improve this lower bound by 1 (see~\cite{SY}), which is tight (see~\cite{GORZ}). 

\bigskip

Let us recall a classic model of random graphs that we study in this paper. The \textbf{binomial random graph} $\G(n,p)$ is the random graph $G$ with vertex set $[n]$ in which every pair $\{i,j\} \in \binom{[n]}{2}$ appears independently as an edge in $G$ with probability~$p$. Note that $p=p(n)$ may (and usually does) tend to zero as $n$ tends to infinity.  

\begin{figure}
\begin{center}
\begin{tabular}{cc}
\includegraphics[scale=0.3]{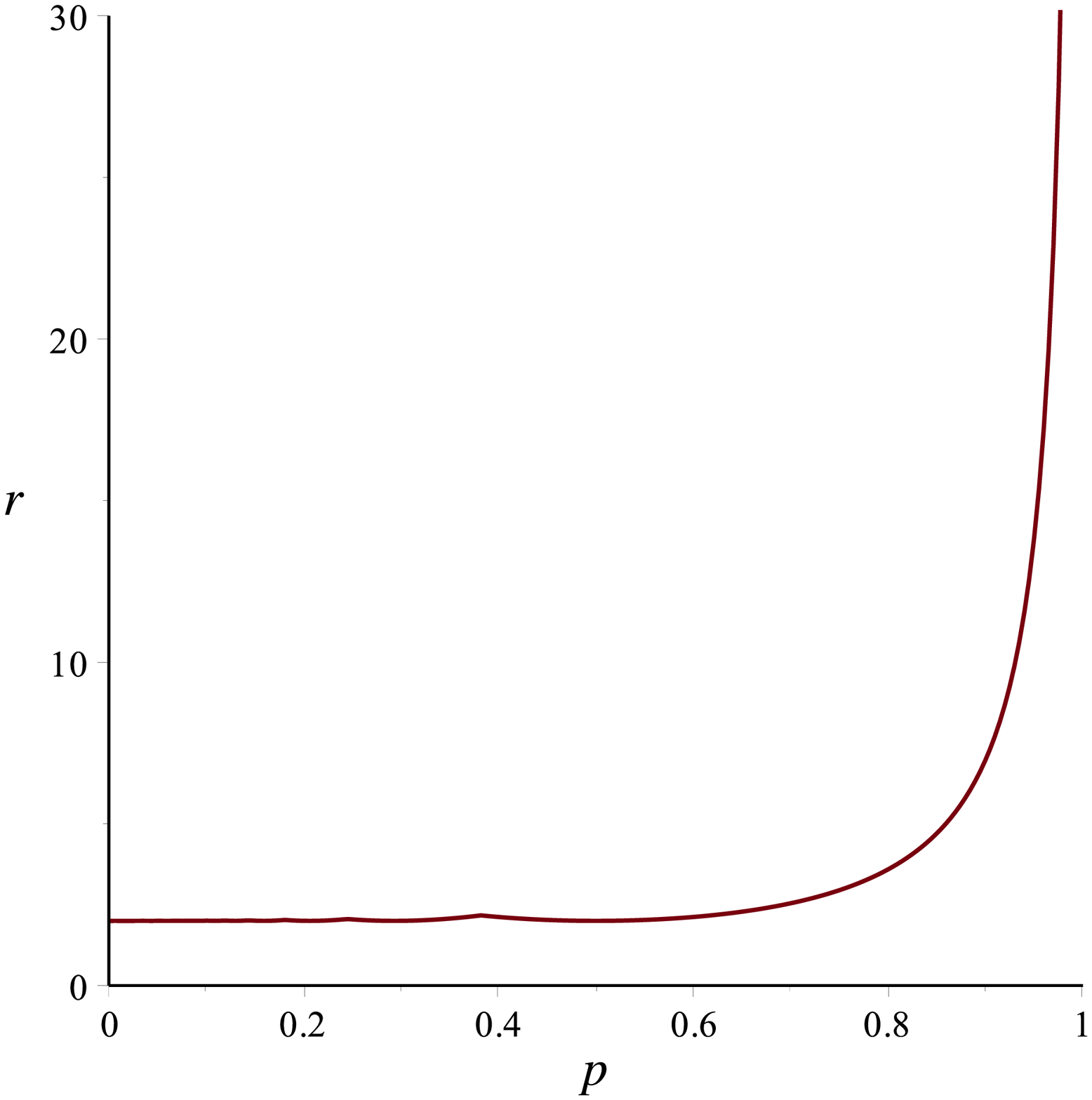}
\qquad
\includegraphics[scale=0.3]{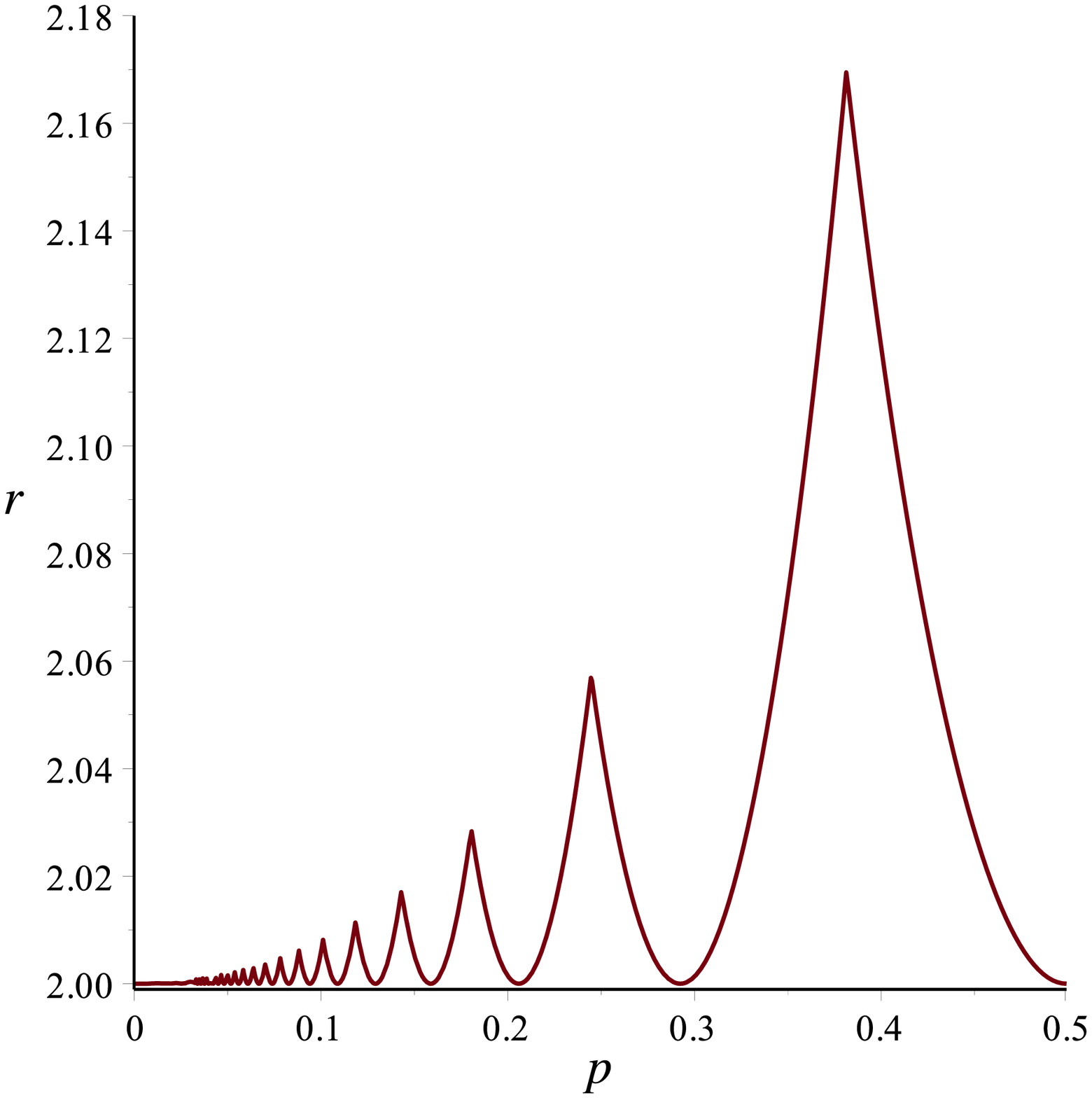} 
\end{tabular}
\end{center}
\caption{The function $r=r(p)$ for $p\in (0,1)$ and $p\in(0,1/2]$, respectively.}\label{fig1}
\end{figure}

All asymptotics throughout are as $n \rightarrow \infty $ (we emphasize that the notations $o(\cdot)$ and $O(\cdot)$ refer to functions of $n$, not necessarily positive, whose growth is bounded). We say that an event in a probability space holds \textbf{asymptotically almost surely} (or \textbf{a.a.s.}) if the probability that it holds tends to $1$ as $n$ goes to infinity. Since we aim for results that hold a.a.s., we will always assume that $n$ is large enough. We often write $\G(n,p)$ when we mean a graph drawn from the distribution $\G(n,p)$.  For simplicity, we will write $f(n) \sim g(n)$ if $f(n)/g(n) \to 1$ as $n \to \infty$ (that is, when $f(n) = (1+o(1)) g(n)$). Finally, we use \textbf{lg} to denote logarithms with base~2 and \textbf{log} to denote natural logarithms.

\bigskip

Before we state the main result of this paper, we need a few definitions that we will keep using throughout the whole paper. For a given $p=p(n)$ satisfying
$$
p \ge \frac {4}{\log 2} \cdot \frac {(\log n)(\log \log n)} {n} \hspace{1cm} \text{ and } \hspace{1cm} p \le 1-\eps
$$
for some $\eps > 0$, let
$$
s = s(p) = \min \Big\{ [(1-p)^\ell]^2 + [1-(1-p)^\ell]^2 : \ell \in \N \Big\},
$$ 
and let $\ell_0$ be a value of $\ell$ that achieves the minimum ($\ell_0$ can be assigned arbitrarily if there are at least two such values). We will show in Section~\ref{sec:upper_bound} that
\begin{equation}\label{eq:lnot}
\ell_0 \in \left\{ \left\lfloor \frac {\log(1/2)}{\log(1-p)} \right\rfloor , \left\lceil \frac {\log(1/2)}{\log(1-p)} \right\rceil \right\},
\end{equation}
and that
\begin{equation}\label{eq:sp}
\frac{1}{2} \le s(p) \le \frac{1+p^2}{2}.
\end{equation}
If $p$ is a constant, then $r=r(p)$ is defined such that $n^2 s^{r \lg n} = 1$, that is,
\begin{equation}\label{eq:r}
r = r(p) = \frac {2}{\lg(1/s)}.
\end{equation}
Observe that $r$ tends to infinity as $p \to 1$ and undergoes a ``zigzag'' behaviour as a function of $p$ (see Figure~\ref{fig1}). The reason for such a behaviour is, of course, that the function $s$ is not monotone (see Figure~\ref{fig2}). Furthermore, observe that for each $p=1-(1/2)^{1/k}$, where $k$ is a positive integer, $\ell_0=k$, $s=1/2$, and $r=2$.

\begin{figure}
\begin{center}
\begin{tabular}{cc}
\includegraphics[scale=0.3]{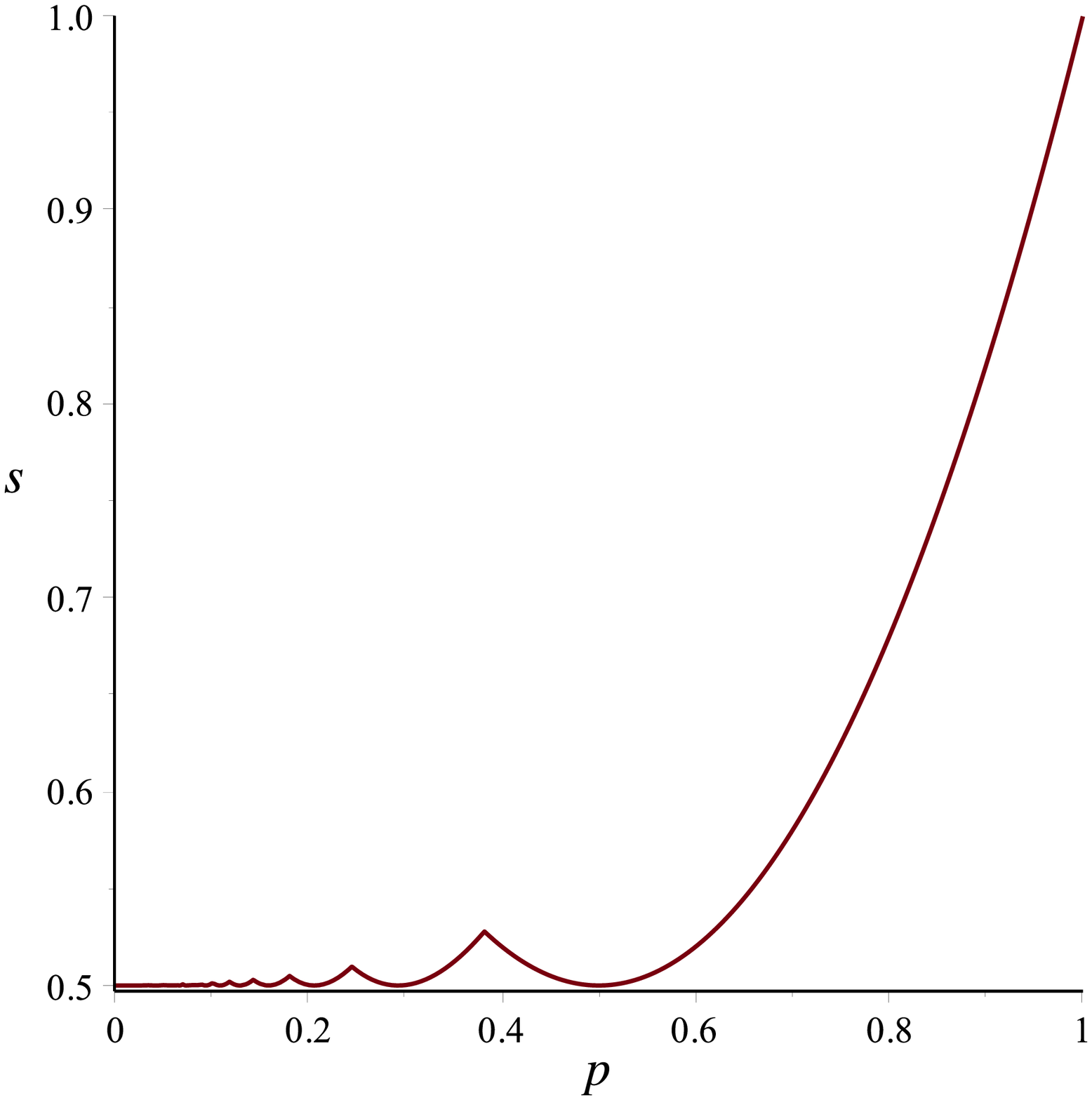}
\qquad
\includegraphics[scale=0.3]{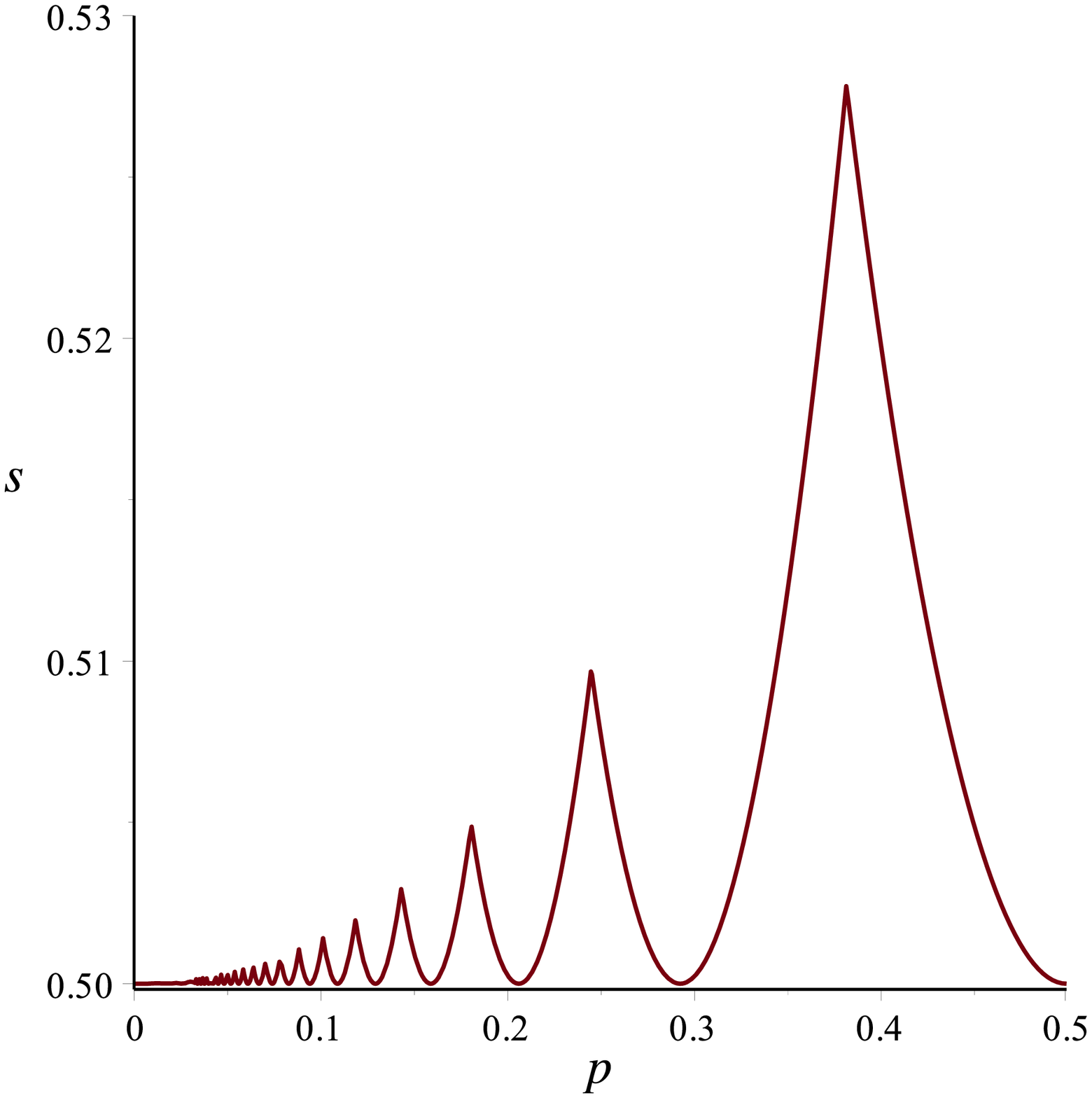} 
\end{tabular}
\end{center}
\caption{The function $s=s(p)$ for $p\in (0,1)$ and $p\in(0,1/2]$, respectively.}\label{fig2}
\end{figure}

\bigskip 

Now we state the main result of the paper.

\begin{theorem}\label{thm:main}
Suppose that $p=p(n)$ is such that 
$$
p \gg \frac {(\log n)^2 (\log (np))^2}{n} \hspace{1cm} \text{ and } \hspace{1cm} p \le 1-\eps,
$$
for some $\eps \in (0,1)$. Let $G \in \G(n,p)$. Then, the following holds a.a.s. 
\begin{enumerate}
\item[(i)] If $p$ is a constant, then
$$
\chi_s (G) \sim r \lg n.
$$
\item[(ii)] If $p=o(1)$ and $np = n^{\alpha+o(1)}$ for some $\alpha \in (0, 1]$, then 
$$
(2\alpha+o(1)) \lg n \le \chi_s (G) \le (1+\alpha+o(1)) \lg n.
$$
\item[(iii)] If $np = n^{o(1)}$, then 
$$
2 (\lg (np)-\lg \log n-\lg \log (np))  \le \chi_s (G) \le (1+o(1)) \lg n.
$$
\end{enumerate}
\end{theorem}

\bigskip

Note that the result is asymptotically tight for dense graphs (that is, for $np = n^{1-o(1)}$; see part (i) and part (ii) for $\alpha=1$). For sparser graphs (part (ii) for $\alpha \in (0,1)$) the ratio between the upper and the lower bound is a constant that gets large for $\alpha$ small. On the other hand, the trivial lower bound of $\lg \chi(G)$ (see \eqref{eq:chis_trivial}) gives us the following: a.a.s.
$$
\chi_s(\G(n,p)) \ge \lg \chi(\G(n,p)) \sim \lg \left( \frac {pn}{2 \log (pn)} \right) \sim \alpha \lg n,
$$
provided that $pn \to \infty$ as $n \to \infty$, and $p=o(1)$; $\chi_s(\G(n,p)) \ge \lg \chi(\G(n,p)) = \Omega(1)$ otherwise (see~\cite{L,M}). So the lower bound we prove is by a multiplicative factor of $2+o(1)$ larger than the trivial one, provided that $\log (np) / \log \log n \to \infty$. If $np = \log^{C+o(1)} n$ for some $C \in [2,\infty)$, then our bound is by a factor of $2(C-1)/C + o(1)$ better than the trivial one. This seemingly small improvement is important to obtain the asymptotic behaviour in the case $\alpha = 1$, and in particular, to obtain the zig-zag for constant $p$.

\bigskip

The upper and the lower bounds are proved in Section~\ref{sec:upper_bound} and Section~\ref{sec:lower_bound}, respectively. Let us also mention that, in fact, the two bounds proved below are slightly stronger. In particular, the upper bound holds for $pn \ge (2/\log 2) (\log n) (\log \log n)$, the point where the trivial bound of $\chi(\G(n,p))$ becomes stronger. 

\section{Preliminaries}

We will use the following version of \textbf{Chernoff's bound}. Suppose that $X \in \Bin(n,p)$ is a binomial random variable with expectation $\mu=np$. If $0<\delta<1$, then 
$$
\Prob [X < (1-\delta)\mu] \le \exp \left( -\frac{\delta^2 \mu}{2} \right),
$$ 
and if $\delta > 0$,
$$
\Prob [ X > (1+\delta)\mu] \le \exp\left(-\frac{\delta^2 \mu}{2+\delta}\right).
$$
These inequalities are well known and can be found, for example, in~\cite{JLR}.

\bigskip

We will also use \textbf{Suen's inequality} that was introduced in~\cite{Suen} and revised in~\cite{Janson_Suen}. 
For a finite set $S$, let $\I = \{ (x,y) : x, y \in S, \, x \neq y \}$, and for any $(x,y) \in \I$, let $A_{x,y}$ be some event with the corresponding indicator random variable $I_{x,y}$. (In our application, $A_{x,y}$ will be the event that vertices $x$ and $y$ have the same neighbourhood colour sets.) Let $X = \sum_{(x,y) \in \I} I_{x,y}$ be the random variable counting how many such events occur. The associated \textbf{dependency graph} has $\I$ as its vertex set, and $(x_1, y_1) \sim (x_2, y_2)$ if and only if $\{ x_1, y_1 \} \cap \{ x_2, y_2 \} \neq \emptyset$. Suen's inequality asserts that
\begin{equation}\label{eq:Suen}
\Prob(X=0) \le \exp \left( - \mu + \Delta e^{2\delta} \right),
\end{equation}
where
\begin{eqnarray*}
\mu &=& \sum_{(x,y) \in \I} \Prob(A_{x,y}),\\
\Delta &=& \sum_{ (x_1, y_1) \sim (x_2, y_2)} \E[I_{x_1,y_1} I_{x_2,y_2}],\\
\delta &=& \max_{(x_1,y_1) \in \I} \sum_{(x_2, y_2) \sim (x_1, y_1) } \Prob(A_{x_2, y_2}).
\end{eqnarray*}

\section{Upper bound}\label{sec:upper_bound}
We start by proving~\eqref{eq:lnot} and~\eqref{eq:sp}. Since
\[
[(1-p)^\ell]^2 + [1-(1-p)^\ell]^2 = 2\left[ (1-p)^{\ell} -1/2\right]^2 + 1/2,
\]
it follows that $s\ge 1/2$, and consequently \eqref{eq:lnot} also holds. Now, let 
$$
\ell = \left\lceil \frac {\log(1/2)}{\log(1-p)} \right\rceil =  \frac {\log(1/2)}{\log(1-p)}  + \delta, 
$$
where $0\le \delta < 1$. Observe that
\[
s(p) \le [(1-p)^\ell]^2 + [1-(1-p)^\ell]^2
=\frac{[1-(1-p)^{\delta}]^2 + 1}{2}
\le \frac{[1-(1-p)]^2 + 1}{2}
= \frac{p^2+1}{2}.
\]
implying the upper bound in~\eqref{eq:sp}.

\bigskip

We keep the definition of function $r=r(p)$ for constant $p$ introduced above (see~(\ref{eq:r})). We extend it here for sparser graphs as follows: suppose that $p$ tends to zero as $n \to \infty$, and that $np = n^{\alpha+o(1)}$ for some $\alpha \in [0, 1]$. Then, we define $r=r(p)$ such that $n^2 p s^{r \lg n} = 1$, that is, 
$$
r = r(p) \sim 1+\alpha,
$$ 
since it follows from~\eqref{eq:sp}  that $s \sim 1/2$. 

\bigskip

The upper bound in Theorem~\ref{thm:main} follows immediately from the next lemma. 

\begin{lemma}\label{lem:upper_bound}
Suppose that $p=p(n)$ is such that 
$$
p \ge \frac {2}{\log 2} \cdot \frac {(\log n)(\log \log n)} {n} \hspace{1cm} \text{ and } \hspace{1cm} p \le 1-\eps,
$$
for some fixed $\eps \in (0,1)$. Let $G \in \G(n,p)$. Then, a.a.s.\  $\chi_s (G) \le (r+o(1)) \lg n.$
\end{lemma}

Before we move to the proof, let us note that the lower bound for $p$ is not necessary, and the result can be extended to sparser graphs. The reason it is introduced here is that for sparser graphs, the trivial upper bound of $\chi(G)$ is stronger; note that  a.a.s.
$$
\chi_s(\G(n,p)) \le \chi(\G(n,p)) \sim \frac {pn}{2 \log (pn)},
$$
provided that $pn \to \infty$ as $n \to \infty$, and $p=o(1)$; $\chi_s(G) \le \chi(G) = O(1)$ otherwise. 

\begin{proof}
The proof is straightforward. Let $\omega = \omega(n) = o(\log n)$ be any function tending to infinity with $n$ (slowly enough). Before exposing the edges of the (random) graph $G$, we partition (arbitrarily) the vertex set into $r \lg n + \omega$ sets, each consisting of $\ell_0$ \textbf{important} vertices, and one remaining set of vertices, these being not important. (For expressions such as $r \lg n + \omega$ that clearly have to be an integer, we round up or down but do not specify which: the choice of which does not affect the argument.) Note that 
$$
(r \lg n + \omega) \ell_0 = O(\log n / p) = O(n / \log \log n) = o(n),
$$
and so there are enough vertices to perform this operation. All vertices in a given set receive the same colour, and hence the total number of colours is equal to $(r+o(1)) \lg n$.

For a given pair of vertices, $x,y$, we need to estimate from above the probability $p(x,y)$ that they have the same neighbourhood colour sets. We do it by considering sets of important vertices that neither $x$ nor $y$ belong to. Let $U$ be the set of (important) vertices of the same colour, and let $\ell_0 = |U|$. Then, either both $x$ and $y$ are not connected to any vertex from $U$, yielding the contribution $[(1-p)^{\ell_0}]^2$ to the probability $p(x,y)$, or both $x$ and $y$ are connected to at least one vertex from $U$, giving the contribution $[1-(1-p)^{\ell_0}]^2$. Thus,
$$
p(x,y) \le \Big( [(1-p)^{\ell_0}]^2 + [1-(1-p)^{\ell_0}]^2 \Big)^{r \lg n + \omega - 2} = s^{r \lg n + \omega - 2}.
$$
Hence, the expected number of pairs of adjacent vertices that are \emph{not} distinguished by their neighbourhood colour sets is at most
$$
{n \choose 2} p s^{r \lg n + \omega - 2} \sim n^2 p s^{r \lg n}  \cdot \frac{s^{\omega-2}}{2} = \frac{s^{\omega-2}}{2},
$$
where the last equality follows from the definition of $r$. Finally, by~\eqref{eq:sp}, we get that $s(p) \le (p^2+1)/2 \le ((1-\eps)^2+1)/2<1$ and so $s^{\omega-2}/2$ tends to zero as $n\to \infty$. Hence, the lemma follows by Markov's inequality.
\end{proof}

\section{Lower bound}\label{sec:lower_bound}

Before we move to the proof of the lower bound, we need the following technical lemma. 
\begin{lemma}\label{lem:ratio}
Let $1/2 \le x, y\le 1$ and $\beta_x, \beta_y$ be any positive real numbers. Then, there exist unique $s$ and $z$ such that $1/2\le s,z \le 1$ and
\begin{equation}\label{eq:lem:ratio:eq}
(x^2 + (1-x)^2)^{\beta_x} (y^2 + (1-y)^2)^{\beta_y} = (z^2 + (1-z)^2)^{\beta_x + \beta_y} = s^{\beta_x + \beta_y}.
\end{equation}
Moreover, 
\begin{equation}\label{eq:lem:ratio:ineq}
\frac {(x^3 + (1-x)^3)^{\beta_x} (y^3 + (1-y)^3)^{\beta_y}}  {(x^2 + (1-x)^2)^{\beta_x} (y^2 + (1-y)^2)^{\beta_y}}
\le \frac {(z^3 + (1-z)^3)^{\beta_x + \beta_y}} {(z^2 + (1-z)^2)^{\beta_x + \beta_y}} 
= \left( \frac {3s-1}{2s} \right)^{\beta_x + \beta_y}.
\end{equation}
\end{lemma}

\bigskip

\noindent
The lemma can be inductively applied to obtain the following corollary.

\begin{corollary}\label{cor:ratio}
Let $1/2 \le x_1, x_2, \dots, x_k \le 1$. Then, there exist unique $s$ and $z$ such that $1/2 \le s,z \le 1$ and
$$
\prod_{i=1}^{k} (x_i^2 + (1-x_i)^2) = (z^2 + (1-z)^2)^k = s^k.
$$
Moreover, 
$$
\frac {\prod_{i=1}^{k} (x_i^3 + (1-x_i)^3)}{\prod_{i=1}^{k} (x_i^2 + (1-x_i)^2)}
 \le \frac {(z^3 + (1-z)^3)^{k}} {(z^2 + (1-z)^2)^k}= \left( \frac {3s-1}{2s} \right)^k.
$$
\end{corollary}

\bigskip

\begin{proof}[Proof of Lemma~\ref{lem:ratio}]
Let $1/2 \le x, y\le 1$ and $\beta_x, \beta_y$ be fixed positive real numbers.  First we show that there exist unique numbers $z$ and $s$ satisfying~\eqref{eq:lem:ratio:eq}. Since $f(t):=t^2+(1-t)^2$ is increasing on $[1/2,1]$, we get 
\[
(f(1/2))^{\beta_x+\beta_y} \le (x^2 + (1-x)^2)^{\beta_x} (y^2 + (1-y)^2)^{\beta_y} \le (f(1))^{\beta_x+\beta_y}.
\]
Clearly, $(f(t))^{\beta_x+\beta_y}$ is also increasing and continuous on $t\in [1/2,1]$, and thus, there is a unique real number $z$ 
such that $1/2\le z \le 1$ and
\[
(x^2 + (1-x)^2)^{\beta_x} (y^2 + (1-y)^2)^{\beta_y} = (f(z))^{\beta_x+\beta_y}.
\]
To finish the proof of~\eqref{eq:lem:ratio:eq}, set $s = z^2+(1-z)^2$ and observe that $1/2 \le s \le 1$, since $1/2 \le z \le 1$.

Now we move to the proof of~\eqref{eq:lem:ratio:ineq}. Let
\[
a = x^2 + (1-x)^2 \quad \text{ and } \quad b=y^2 + (1-y)^2.
\]
Since for every real number $t$, 
\begin{equation}\label{eq:lem:ratio:t}
t^3 + (1-t)^3 = \frac{3(t^2+(1-t)^2) - 1}{2},
\end{equation}
we get
\[
x^3 + (1-x)^3 = \frac{3a-1}{2} \quad \text{ and } \quad  y^3 + (1-y)^3 = \frac{3b-1}{2}.
\]
Furthermore, 
\[
z^2 + (1-z)^2 = (x^2 + (1-x)^2)^{\frac{\beta_x}{\beta_x + \beta_y}} (y^2 + (1-y)^2)^{\frac{\beta_y}{\beta_x + \beta_y}} = a^{\frac{\beta_x}{\beta_x + \beta_y}} b^{\frac{\beta_y}{\beta_x + \beta_y}},
\]
and so
\[
z^3 + (1-z)^3 = \frac{ 3a^{\frac{\beta_x}{\beta_x + \beta_y}} b^{\frac{\beta_y}{\beta_x + \beta_y}}    -1}{2}.
\]
In order to show the inequality in~\eqref{eq:lem:ratio:ineq} it suffices to prove
\[
(z^3 + (1-z)^3)^{\beta_x + \beta_y} \ge (x^3 + (1-x)^3)^{\beta_x} (y^3 + (1-y)^3)^{\beta_y},
\]
which is equivalent to 
\[
\left( \frac{ 3a^{\frac{\beta_x}{\beta_x + \beta_y}} b^{\frac{\beta_y}{\beta_x + \beta_y}}    -1}{2} \right)^{\beta_x + \beta_y}
\ge \left( \frac{3a-1}{2} \right)^{\beta_x} \left( \frac{3b-1}{2} \right)^{\beta_y}
\]
and subsequently to 
\begin{equation}\label{eq:hoelder1}
(3a)^{\frac{\beta_x}{\beta_x + \beta_y}} (3b)^{\frac{\beta_y}{\beta_x + \beta_y}}    -1
\ge (3a-1)^{\frac{\beta_x}{\beta_x+\beta_y}} (3b-1)^{\frac{\beta_y}{\beta_x+\beta_y}}.
\end{equation}

Set 
\[
p = \frac{\beta_x + \beta_y}{\beta_x} \quad \text{ and } \quad q = \frac{\beta_x+\beta_y}{\beta_y}.
\]
Now, showing~\eqref{eq:hoelder1} (and hence also \eqref{eq:lem:ratio:ineq}) is equivalent to showing
\[
(3a)^{\frac{1}{p}} (3b)^{\frac{1}{q}}  
\ge (3a-1)^{\frac{1}{p}} (3b-1)^{\frac{1}{q}} + 1.
\]
The latter inequality immediately from H\"older's inequality (see, for example, \cite{HLP}): indeed, let
\[
a_1 = (3a-1)^{\frac{1}{p}}, \quad b_1 = (3b-1)^{\frac{1}{q}} \quad \text{ and } \quad a_2=b_2=1.
\]
(Observe that $a_1$ and $b_1$ are well-defined since $3a-1>0$ and $3b-1>0$.)
Then, since $1/p + 1/q = 1$, $p>1$, and $q>1$, H\"older's inequality yields
\[
(a_1^p + a_2^p)^{\frac{1}{p}} (b_1^q + b_2^q)^{\frac{1}{q}}
\ge a_1b_1+ a_2b_2,
\]
as required.
Finally note that the equality in~\eqref{eq:lem:ratio:ineq} follows from~\eqref{eq:lem:ratio:t} applied with $s = z^2 + (1-z)^2$. The proof of the lemma is finished.
\end{proof}

\bigskip

As we did in the previous section, we keep the definition of the function $r=r(p)$ for constant $p$ (see~(\ref{eq:r})). We extend it here for sparser graphs as follows (in a different way than in the previous section): suppose that $p$ tends to zero as $n \to \infty$ and that $np = n^{\alpha+o(1)}$ for some $\alpha \in [0, 1]$. This time, $r=r(p)$ is defined such that $(np)^2 s^{r \lg n} = (\log^2 n)(\log^2 (np))$, that is, 
$$
r = r(p) = \frac {-2 (\lg (np)-\lg \log n-\lg \log (np))}{(\lg n) (\lg s)} \ge \frac {2 (\lg (np)-\lg \log n-\lg \log (np))}{\lg n} \sim 2\alpha,
$$ 
since $s\ge 1/2$.
 
\bigskip

Now we are ready to come back to the proof of the lower bound. The lower bound in Theorem~\ref{thm:main} follows immediately from the following lemma. 

\begin{lemma}\label{lem:lower_bound}
Suppose that $p=p(n)$ is such that 
$$
p \gg \frac{(\log n)^2(\log (np))^2}{n} \hspace{1cm} \text{ and } \hspace{1cm} p \le 1-\eps,
$$
for some fixed $\eps \in (0, 1)$. Let $G \in \G(n,p)$. Then, a.a.s.\ $\chi_s (G) \ge r \lg n,$ provided that $p=o(1)$, and $\chi_s (G) \ge (1+o(1)) r \lg n$ otherwise. 
\end{lemma} 

\begin{proof}
First, let us note that, since the expected degree tends to infinity faster than $(\log^2 n)(\log^2(np))$, it follows immediately from Chernoff's bound and the union bound that a.a.s.\ all vertices have degree at most, say, $2pn$. Hence, since we aim for a statement that holds a.a.s., we may assume that the maximum degree of $G$ is at most $2pn$. In fact, the argument is slightly more delicate and will be explained soon. 
\medskip

Suppose that we are given a colouring of the vertices. We partition all colours into \textbf{important} and \textbf{unimportant} ones: a colour is important if the number of vertices of that colour is at most $2 \log n / p$. 
First, let us show that unimportant colours can distinguish only a few edges. Formally, we claim the following:\\ \\ 
\textbf{Claim:} A.a.s.\ each set of $2 \log n / p$ vertices dominates all but at most $2 \log n / p$ vertices. \\ \\
\textit{Proof of the claim.}  Note that the expected number of pairs of disjoint sets of size $2 \log n / p$ with no edge between them is at most
\begin{eqnarray*}
{ n \choose \frac {2 \log n}{p} }^2 (1-p)^{ (2 \log n / p)^2 } &\le& \left( \frac {nep}{2 \log n} \right)^{4 \log n / p} \exp \left( - \frac {4 \log^2 n}{p} \right)\\
&=& o \left( n^{4 \log n / p} \exp \left( - \frac {4 \log^2 n}{p} \right) \right) = o(1).
\end{eqnarray*}
The claim follows from the first moment method. $\hfill \square$ 

\medskip

Hence, if $p$ is  constant, then $O(\log n)$ unimportant colours can distinguish only $O(\log^2 n) = o(n)$ vertices. All remaining vertices will have \emph{all} unimportant colours present in their neighbourhood colour sets; as a result, no edge in the graph induced by these vertices can be distinguished by unimportant colours. On the other hand, if $p=o(1)$, then at most $(2+o(1)) \lg (np) \leq (2+o(1))\lg n$ unimportant colours can distinguish at most $6 \log^2 n / p = o(n)$ vertices, since $pn \gg \log^2 n$. As a consequence of the claim, we may concentrate on important colours from now.

\bigskip

Suppose that a colouring $c : I \to [k]$ using $k$ important colours is fixed; $I \subseteq V$ is the set of vertices coloured with important colours. Moreover, let us fix a set $U \subseteq V$ of $O( \log^2 n / p) = o(n)$ vertices that are (possibly) distinguished by unimportant colours. Our goal is to estimate the probability $q(c,U)$ that important colours distinguish endpoints of edges in $G[V \setminus (I \cup U)]$, which is the graph induced by those vertices that are coloured with unimportant colours and that are adjacent to at least one vertex from each unimportant colour class. Since the number of configurations to investigate is at most 
$$
{ n \choose \frac {2 \log n}{p} }^{O(\log np)} { n \choose O( \frac {\log^2 n}{p} ) } \le \left( np \right)^{O(\log^2 n/ p)} = \exp \left( O \left( \frac {(\log n)^2 (\log (np)) }{p} \right) \right),
$$
it is enough to estimate $q(c,U)$ from above by, say, 
\begin{equation}\label{eq:union}
Q = Q(p) := \exp( - (\log^2 n)(\log^{3/2} (np)) / p ).
\end{equation}
The result will then follow immediately by the union bound.

\medskip

The expected number of edges in $G[V \setminus (I \cup U)]$ is ${(1-o(1)) n \choose 2} p \ge n^2 p / 3 \gg n \log^2 n$, and so it follows from Chernoff's bound that with probability at most $\exp( - n \log^2 n) \le Q/2$ the number of edges is smaller than, say, $n^2 p / 4$. On the other hand, if the maximum degree in $G[V \setminus (I \cup U)]$ is larger than $2pn$ (for some configuration $(c,U)$), then we stop the whole argument and claim no lower bound for $\chi_s(G)$. Recall that at the beginning of the proof, we showed that a.a.s.\ $\Delta(G) \le 2pn$. Clearly, if this is the case, then (deterministically) the degree of each vertex in $G[V \setminus (I \cup U)]$ is at most $2pn$. Hence, we may condition on the event that the graph $G[V \setminus (I \cup U)]$ has the following two properties: (i) the number of edges is at least $n^2p/4$, and (ii) no vertex has degree more than $2pn$. It is important that no edge between $V \setminus (I \cup U)$ and $I$ has been exposed yet.

\bigskip

Let us focus on constant $p$ first. Suppose that the number of important colours is equal to 
$$
k := r \lg n + \frac {5 \log \log n}{\log s} = r \lg n - O(\log \log n) \sim r \log n,
$$
where the error term follows from~\eqref{eq:sp} and from our assumption that $p \leq 1-\eps$ from which we get, as before, $s \leq ((1-\eps)^2 + 1)/2 < 1$. Suppose that for a given configuration $(c,U)$, the probability $p(x,y)$ that two adjacent vertices $x,y$ from $V \setminus (I \cup U)$ are not distinguished by important colours is $t^k$ for some $t \ge s$. (Recall that $s^k$ is the lower bound for $p(x,y)$ which can be attained when all colour classes have size $\ell_0$.) We will use Suen's inequality to obtain an upper bound for $q(c,U)$. Let 
$$
\I = \{ (x,y) : x, y \in V \setminus (I \cup U), \, x \neq y, xy \in E \}. 
$$
For any $(x,y) \in \I$, let $A_{x,y}$ be the event (with the corresponding indicator random variable $I_{x,y}$) that the neighbourhood colour sets (restricted to important colours only) of $x$ and $y$ are equal. Let $X = \sum_{(x,y) \in \I} I_{x,y}$. We wish to estimate the probability that $X=0$. Denote by $\kappa_i$ (for $1\le i\le k$) the number of vertices coloured by colour~$i$.

\medskip

Suppose first that $t=s$. This means
\begin{equation}\label{eq:Axy}
\Prob(A_{x,y}) = \prod_{i=1}^k \left( [(1-p)^{\kappa_i}]^2 + [1-(1-p)^{\kappa_i}]^2 \right) = s^k.
\end{equation}
In this case,
$$
\mu = s^k \cdot | \I | \ge s^k \cdot \frac {n^2 p}{4} = s^{r \lg n} \log^5 n \cdot \frac {n^2 p}{4} =  \frac {p \log^5 n}{4} \gg \log^4 n,
$$
where the last equality follows from the definition of~$r$.  Observe that for $(x,y)$ and $(y,z)$ in $\I$ we get
\[
\Prob(A_{x,y}A_{y,z}) = \prod_{i=1}^k \left( [(1-p)^{\kappa_i}]^3 + [1-(1-p)^{\kappa_i}]^3 \right).
\]
Thus, since the number of pairs $(x,y)$ and $(y,z)$ is bounded by $| \I | \cdot 4pn$, we have
\begin{align*}
\Delta &\le | \I | \cdot 4pn \cdot \prod_{i=1}^k \left( [(1-p)^{\kappa_i}]^3 + [1-(1-p)^{\kappa_i}]^3 \right)\\
&= | \I | \cdot s^k \cdot 4pn \cdot \frac{\prod_{i=1}^k \left( [(1-p)^{\kappa_i}]^3 + [1-(1-p)^{\kappa_i}]^3 \right)}{\prod_{i=1}^k \left( [(1-p)^{\kappa_i}]^2 + [1-(1-p)^{\kappa_i}]^2 \right)}\\
&=\mu \cdot 4pn \cdot \frac{\prod_{i=1}^k \left( [(1-p)^{\kappa_i}]^3 + [1-(1-p)^{\kappa_i}]^3 \right)}{\prod_{i=1}^k \left( [(1-p)^{\kappa_i}]^2 + [1-(1-p)^{\kappa_i}]^2\right) }.
\end{align*}
Now, using~\eqref{eq:Axy}, we apply Corollary~\ref{cor:ratio} with $x_i = (1-p)^{\kappa_i}$ if $(1-p)^{\kappa_i}>1/2$, and $x_i = 1-(1-p)^{\kappa_i}$ otherwise, to get
\[
\Delta \le \mu \cdot 4pn \cdot \left( \frac {3s-1}{2s} \right)^k.
\]
Now we will prove that $\Delta\ll \mu$. Using Taylor expansion at $s=1$, one can show that for any $s \in [1/2,1]$ we have
\[
\frac {\left(\frac {3s-1}{2s} \right)}{\sqrt{s}} = \frac {3s-1}{2s^{3/2}} = 1 - \frac 38 (1-s)^2 - \frac 58 (1-s)^3 - \frac {105}{128} (1-s)^4 + \ldots \le 1 - \frac 38 (1-s)^2 .
\]
Furthermore, since \eqref{eq:sp} together with $p\le 1-\eps$ implies
\[
s \le \frac{p^2+1}{2} \le \frac{p+1}{2} \le 1-\frac{\eps}{2},
\]
we get $\left(\frac {3s-1}{2s} \right) / \sqrt{s} \le 1 - 3\eps^2/32$, and hence
\begin{eqnarray*}
\Delta &\le& \mu \cdot (4np) s^{k/2} \left( 1 - \frac{3}{32} \eps^2 \right)^k \\
&=& \mu \cdot (4np) \Big( s^{(r \lg n)/2} \log^{5/2} n \Big) n^{-\Omega(\eps^2)} \\
&=& \mu \cdot (4p \log^{5/2} n) n^{-\Omega(\eps^2)} \ll \mu.
\end{eqnarray*}
Finally, 
$$
\delta \le 2 (2pn) s^k = 4pn^{-1} \log^5 n = o(1).
$$
It follows from Suen's inequality (see~(\ref{eq:Suen})) that 
$$
\Prob(X=0) \le \exp \left( - \mu + \Delta e^{2\delta} \right) = \exp (- (1+o(1)) \mu) \le Q/2,
$$
and $q(c,U) \le Q/2 + \Prob(X=0) \le Q$, as needed (see~\eqref{eq:union}). 

\medskip

Suppose now that $t>s$. In this case, $\mu$ is larger than before but, unfortunately, $\Delta$ grows faster than $\mu$ (as $t$ grows) and eventually becomes larger than $\mu$. In order to avoid this undesired situation, we make the dependency graph sparser so that $\mu$ is still of order $\log^5 n$. Let us note that if $t \ge u$, where $u$ is defined so that $(u/s)^k = pn$, then Suen's inequality can be avoided and we can simply use Chernoff's bound to obtain the desired upper bound for $\Prob(X=0)$: indeed, it follows from the claim proved above that there exists a matching in $G[V \setminus (I \cup U)]$ consisting of at least $n/3$ edges; otherwise, the remaining $n-O(\log^2 n/p)-2n/3=n/3-o(n) \gg \log n / p$ vertices in $V \setminus (I \cup U)$ would form an independent set that could be split into two sets of equal size and, clearly, no edge will be present between them, contradicting the claim. Since the events are now independent, and the expected number of edge endpoints not distinguished is
$$
\mu \ge \frac {t^k n}{3} = \frac {s^k (t/s)^k n}{3} \ge \frac {s^k n^2 p}{3} = \frac {p \log^5 n}{3} \gg \log^4 n,
$$
the desired bound holds, since $\Prob(X=0) \le \exp(-\Omega(\mu))$ by Chernoff's bound.

\medskip

It remains to consider the case $s < t < u$. Our goal is to scale the degree in the dependency graph down by a multiplicative factor of 
$$
\xi = \xi(t) := (s/t)^k \ge (s/u)^k = 1/(pn). 
$$
Let $\I_{\xi}$ be a random subgraph of $\I$: each $(x,y) \in \I$ is  independently put into $\I_{\xi}$ with probability $\xi$. Since $|\I| \ge n^2 p / 4$, $\E \, |\I_{\xi}| \ge \xi n^2 p / 4 \ge n/4$,  and thus a.a.s.\ $|\I_{\xi}| \ge \xi n^2 p / 5$. Moreover, since the maximum degree in the dependency graph is at most $2pn$, by Chernoff's bound together with a union bound over all vertices, it follows that a.a.s.\ the maximum degree in a random subgraph of it is at most 
$$
\max \{ 4 \xi pn , 10 \log n \} \le  40 \xi pn \log n.
$$
Therefore, the deterministic (non-constructive) conclusion is that there exists a subgraph of the dependency graph with at least $\xi n^2 p / 5$ pairs and the maximum degree at most $40 \xi pn \log n$. We restrict ourselves to this subgraph, stressing one more time that no edge between $V \setminus (I \cup U)$ and $I$ is exposed yet. Now,
$$
\mu = t^k \cdot | \I_{\xi} | \ge s^k \left( \frac ts \right)^k \cdot \frac {\xi n^2 p}{5} = \frac {p \log^5 n}{5} \gg \log^4 n.
$$
Moreover,
\begin{eqnarray*}
\Delta &\le& t^k \left( \frac {3t-1}{2t} \right)^k \cdot | \I_{\xi} | \cdot 2(40 \xi pn \log n) \le \mu \cdot (80 \xi pn \log n) \left( \frac {3t-1}{2t} \right)^k \\
&=& \mu \cdot (80 pn \log n) \left( \frac {3t-1}{2t^2} \right)^k s^k. 
\end{eqnarray*}
Let $h(t) := (3t-1)/(2t^2)$. Note that $h(1/2)=1$, $h(t)$ is increasing on the interval $[1/2,2/3]$ attaining $h(2/3)=9/8$, and then is decreasing on the interval $[2/3,1]$ going back to $h(1)=1$. Hence, if $s \le 2/3$, then $(h(t)s)^k$, as a function of $t$, is maximized for $t=2/3$. Then, as a function of $s$, since $s \le 2/3$, $r=2/\lg(1/s)$, and $k \sim r \lg n$, $(h(2/3)s)^k$ is maximized for $s=2/3$. We are back to the case $t=s$ that we already checked. On the other hand, if $s > 2/3$, then, since $t \ge s$, $h(t)$ and therefore $\Delta$ is maximized again for $t=s$. Hence, in both cases we have $\Delta \ll \mu.$ Finally,
$$
\delta \le 2 (40 \xi pn \log n) t^k = 80 pn (\log n) s^k = 80 pn^{-1} \log^6 n = o(1).
$$
Hence, Suen's inequality can be applied as before, and the proof for constant $p$ is finished.

\bigskip

The case $p=o(1)$ can be verified exactly the same way. In fact, it is slightly easier since $s = 1/2 + O(p^2)$, $r \le 2+o(1)$, and we do not have to worry about an increasing value of $s$ (and therefore, neither about an increasing value of $r$). We point out only the adjustments of the proof. 
Recall that the definition of $r$ is extended to the case $p=o(1)$. The number of colours is now $k=r \lg n$, and we have $(np)^2 s^k = (\log^2 n)(\log^2 (np))$. 

For the case $t=s$, we have
$$
\mu =s^k |\I| \ge s^k \cdot \frac {n^2 p}{4} = \frac {(\log^2 n)(\log^{2} (np))}{4 p} \gg \frac {(\log^2 n)(\log^{3/2} (np))}{p},
$$
as desired for the union bound (see~\eqref{eq:union}). Since now $s = 1/2 + O(p^2)$, the argument for $\Delta$ is much easier:
\begin{eqnarray*}
\Delta &\le& s^k \left( \frac 12 + O(p^2) \right)^k |\I|\cdot 4pn \\
&\le& \mu (4 np) \left( s + O(p^2) \right)^k \\
&=& 4 \mu np s^k \exp(O(p^2 \log n)) \\
&=& 4 \mu (np)^{-1} (\log^2 n) (\log^2 (np)) \exp(O(p^2 \log n)) \\
&\ll& \mu, 
\end{eqnarray*}
since $np \gg (\log^2 n) (\log^2 (np))$. (Note that for $p=\Omega(1/\sqrt{\log n})$ but still $p=o(1)$, we have $\exp(O(p^2 \log n))=n^{o(1)}$. However, this causes no problem, as $(np)^{-1} = n^{-1+o(1)}$ and so $\Delta \ll \mu$. Otherwise, that is, if $p = o(1/\sqrt{\log n})$, we have $\exp(O(p^2 \log n)) \sim 1$ and $\Delta \ll \mu$ follows easily.) Finally, as before, and again using the same lower bound on $np$, we have 
$$
\delta \le 2 (2pn) s^k =  4 (np)^{-1} (\log^2 n) (\log^2 (np))=o(1),
$$
and Suen's inequality can be applied.

Now let us consider the case $t > s$. The definition of $u$ is not affected, and for $t \ge u$ we use Chernoff's bound since the events are independent. The only difference is that the new value of $s^k$ has to be used to get
$$
\mu \ge \frac {t^k n}{3} = \frac {s^k (t/s)^k n}{3} \ge \frac {s^k n^2 p}{3} = \frac {(\log^2 n)(\log^2 (np))}{3p} \gg \frac {(\log^2 n)(\log^{3/2} (np))}{p},
$$
as desired. 

For the case $s < t < u$, the definition of $\xi$ remains the same and again, after adjusting the value of $s^k$ we get $\mu \gg (\log^2 n)(\log^{3/2} (np))/p$. The argument for $\Delta \ll \mu$ is not affected. Finally,
$$
\delta \le 2 (40 \xi pn \log n) t^k = 80 pn (\log n) s^k = 80 (pn)^{-1} (\log^3 n)(\log^2 (np)) = o(1),
$$
provided that $pn \gg (\log^3 n) (\log^2 (np))$. For slightly sparser graphs, that is, when $(\log^2 n) (\log^2 (np)) \ll pn = O( (\log^3 n) (\log^2 (np)) )$, observe that we only have that $\delta = o(\log n)$, but in fact we can show a stronger bound for $\Delta$: it follows that
$$(h(2/3)s)^k=(9/8)^{2 \lg n/\lg(3/2)} n^{-2} \log^5 n=n^{2 \lg(9/8)/\lg(3/2)-2}\log^5 n,
$$
and thus
$$
 (80 pn \log n) (h(t)s)^k =O(n^{2\lg(9/8)/\lg(3/2)-2}\log^{9} n \log \log n)=n^{-\Omega(1)}.
$$
Hence, $\Delta = \mu n^{-\Omega(1)}$, and so $\Delta e^{2 \delta} \le \mu n^{-\Omega(1)} n^{o(1)} = o(\mu)$, as needed for Suen's inequality to be useful. The proof is finished.
\end{proof}

\end{document}